\documentclass{birkjour}
 \usepackage{mathrsfs}

 \newtheorem{thm}{Theorem}[section]
 \newtheorem{cor}[thm]{Corollary}
 \newtheorem{lem}[thm]{Lemma}
 
 \theoremstyle{definition}
 \newtheorem{defn}[thm]{Definition}
 \theoremstyle{remark}
 \newtheorem{rem}[thm]{Remark}
 
 \newtheorem{qu}{Question}
 
 \numberwithin{equation}{section}

\newcommand{\N}{\mathbb{N}}

\newcommand{\norm}[1]{\left\| #1 \right\|}
\newcommand{\TT}[4]{ T^{#1,#2}_{#3,#4}}
\newcommand{\TTn}[5]{ T\left(#5\right)^{#1,#2}_{#3,#4}}

\allowdisplaybreaks

\begin{document}

%
%
%
%
%
%
%
%
%

\title[The ball-covering property of non-commutative spaces of operators]
 {The ball-covering property of non-commutative spaces of operators on Banach spaces}

\author[Qiyao Bao]{Qiyao Bao}

\address{%
School of Mathematical Sciences and LPMC\\
Nankai University\\
Tianjin\\
China}

\email{qybao@mail.nankai.edu.cn}

\thanks{This work was partially supported by the National Natural Science Foundation of China (No. 11971348, 12071230 and 12471131).}
\author{Rui Liu}
\address{School of Mathematical Sciences and LPMC\br
Nankai University\br
Tianjin\br
China}
\email{ruiliu@nankai.edu.cn}

\author{Jie Shen}
\address{School of Mathematical Sciences and LPMC\br
Nankai University\br
Tianjin\br
China}
\email{1710064@mail.nankai.edu.cn}



\begin{abstract}
A Banach space is said to have the ball-covering property (BCP) if its unit sphere can be covered by countably many closed or open balls off the origin. Let $X$ be a Banach space with a shrinking $1$-unconditional basis. In this paper, by constructing an equivalent norm on $B(X)$, we prove that the quotient Banach algebra $B(X)/K(X)$ fails the BCP. In particular, the result implies that the Calkin algebra $B(H)/ K(H)$,  $B(\ell^p)/K(\ell^p)$ ($1 \leq p <\infty$) and  $B(c_0)/K(c_0)$ all fail the BCP. We also show that $B(L^p[0,1])$ has the uniform ball-covering property (UBCP) for $3/2< p < 3$.
\end{abstract}

\maketitle
\section{Introduction}
The ball-covering property was firstly introduced by Cheng \cite{C2006} and was studied widely by many authors from different perspectives. Almost all properties of Banach spaces can be considered as corresponding properties on the unit sphere of the space, including separability, completeness, reflexivity, smoothness, Radon-Nikodym property \cite{CWZ2011}, uniform convexity, uniform non-squareness \cite{CLL2010}, strict convexity and dentability \cite{SC2015,SC2018}, and universal finite representability and B-convexity \cite{Z2012}. The notion of the ball-covering property plays an important role in the study of geometric and topological properties of Banach spaces
\cite{ AG2023,CKZ2020,CLL2023,FZ32009,FR2016,S2021}. The definition of the ball-covering property is as follows.

\begin{defn}
Let $X$ be a normed space and let $S_{X}$ denote its unit sphere. If there exists a sequence of open balls $(B(x_{n},r_{n}))_{n=1}^{\infty}$ in $X$ such that
$$S_{X}\subseteq\bigcup_{n=1}^{\infty}B(x_{n},r_{n})$$ and $0\notin B(x_{n},r_{n})$, then we say $X$ has the ball-covering property (BCP, in short).
\end{defn}
The centers of the balls are called the {\it BCP points} of $X$. If $X$ has the BCP and the radii of $(B(x_{n},r_{n}))_{n=1}^{\infty}$ are bounded, then $X$ is said to have the {\it strong ball-covering property (SBCP)} \cite{LZ2021}. Moreover, if $X$ has the SBCP, and there exists $r>0$ such that $B(x_{n},r_{n}) \cap B(0,r) = \emptyset$ for all $n\in\mathbb{N_+}$, then $X$ is said to have the {\it uniform ball-covering property (UBCP)} \cite{LZ2021}.

The definition of the BCP shows that all separable normed spaces have the BCP, but the converse is not true \cite{C2006,CCL2008}. In \cite{C2006}, Cheng proved that the non-separable space $\ell^{\infty}$ has the BCP. In \cite{CCL2008}, Cheng et al. showed that $\ell^{\infty}$ can be renormed such that the renormed space fails the BCP, which implies that the BCP is not heritable by its closed subspaces and is not preserved under linear isomorphisms and quotient mappings. Therefore $\ell^{\infty}/ c_{0}$ fails the BCP. Recently, Liu et al. \cite{LLLZ2022} investigated the BCP from commutative function space to non-commutative spaces of operators. They gave a topological characterization of the BCP and showed that the BCP is not hereditary for $1$-complemented subspaces. They proved that the continuous function space $C_{0}(\Omega)$ has the BCP if and only if $\Omega$ has a countable $\pi$-basis where $\Omega$ is a locally compact Hausdorff space. Moreover, they showed that $B(c_{0})$, $B(\ell^{1})$ and every subspace containing finite rank operators in $B(\ell^{p})$ for $1 <p <\infty$ all have the BCP. They also presented some necessary conditions for the bounded linear operators space $B(X,Y)$ to have the BCP. These results established a non-commutative version of Cheng’s result.

Let $H$ be an infinite-dimensional Hilbert space, denote all the bounded linear operators from $H$ to $H$ by $B(H)$. Let $K(H)$ be the ideal of compact operators in $B(H)$, the quotient algebra $B(H)/ K(H)$ is called the Calkin algebra \cite{F2019}. The Calkin algebra is the non-commutative analog of $\ell^{\infty}/ c_{0}$. The following natural question about the Calkin algebra is still open.

\begin{qu}\label{question1}
Does the Calkin algebra $B(H)/ K(H)$ have the BCP?
\end{qu}

Let $X$ be a Banach space, denote all the bounded linear operators from $X$ to $X$ by $B(X)$. Let $K(X)$ be the ideal of compact operators in $B(X)$, a more general question is as follows.

\begin{qu}\label{question2}
Does the quotient Banach algebra $B(X)/ K(X)$ have the BCP?
\end{qu}

In this paper, we give a negative answer to Question \ref{question1} and give a negative answer to Question \ref{question2} when $X$ is a Banach space with a shrinking $1$-unconditional basis by constructing an equivalent norm on $B(X)$.

Let $X$ be a Banach space with a shrinking $1$-unconditinal basis, we fix a real number $\alpha\in [0,1]$, and define a new norm by  $$\|\cdot\|_{\alpha}=\alpha\|\cdot\|_{B(X)}+(1-\alpha)\|\cdot\|_{B(X)/ K(X)}.$$
Then our first main result which characterizes the BCP of the renormed space  $X_\alpha=(B(X),\|\cdot\|_{\alpha})$ is as follows.

\begin{thm}
  Let $X$ be a Banach space with a shrinking $1$-unconditional basis  $(e_n)_{n=1}^\infty$. Then the renormed space $X_{\alpha}=(B(X), \|\cdot\|_{\alpha})$ has the BCP if and only if $\alpha > 1/2$.
\end{thm}

The BCP is a geometric property which has deep connection with the weak star topology for dual spaces. Cheng et al. \cite{CKWZ2010} proved that the BCP does not pass to subspaces. For example, $\ell^{1}[0,1]$ is a subspace of $\ell^{\infty}$ which does not have the BCP. This shows that the $w^{\ast}$-separability of the unit sphere of the dual space $X^{\ast}$ does not imply the BCP of $X$. Cheng \cite{C2006} showed that if $X$ is a Gateaux differentiability space, then $X$ has the BCP if and only if its dual $X^{\ast}$ is $w^{\ast}$-separable, which implies that the BCP is topological invariant among the Gateaux differentiability space. Shang and Cui \cite{SC2013} proved that if a separable space $X$ has the Radon-Nikodym property, then $X^{\ast}$ has the BCP. Fonf and Zanco \cite{FZ2004,FZ12009,FZ22009,FZ32009} investigated the locally finite coverings of the Banach spaces and characterized the relationship between the separability of the dual space and the BCP of $X$. In fact, the BCP only implies the $w^{\ast}$-separability of the dual space \cite{C2006} and every Banach space with a $w^{\ast}$-separable dual can be $(1+\varepsilon)$-renormed to have the SBCP for any $\varepsilon>0$ \cite{CSZ2009,FZ32009}. Luo et al. \cite{LLW2017} showed that for $1 \leq p \leq\infty$, the product space $(X \times Y, \|\cdot\|_{p})$ has the BCP if and only if $X$ and $Y$ have the BCP. Luo and Zheng \cite{LZ2020} proved that for a sequence of normed spaces $\{X_{k}\}$, the direct sum space $(\sum\oplus X_{k})_{\ell^{\infty}}$ has the BCP if and only if every normed space $X_{k}$ has the BCP. They also showed that the dense subspaces preserve the BCP. These results display the stability of the BCP.

In \cite{LZ2020}, Luo and Zheng proved that $L^{\infty}(0,1)$ fails the BCP and if $(\Omega, \Sigma, \mu)$ is a separable measure space, then the space of Bochner integrable functions $L^{p}(\mu, X)$ has the BCP if and only if $X$ has the BCP. 
They also showed that for a separable Lorentz sequence space $E = d(w, p)$, $1 \leq p< \infty$ (or a separable Orlicz sequence space with the $\triangle_{2}$-condition) and a sequence of normed spaces, the space $(\sum\oplus X_{k})_{E}$ has the BCP if and only if all $X_{k}$ have the BCP. In \cite{LZ2021}, Luo and Zheng proved that the SBCP and the UBCP for a Banach space $X$ can be passed to $\ell^{p}(X)$ and $L^{p}([0, 1], X)$ for $1 \leq p \leq \infty$. They showed that $L^{p}([0, 1], X)$ has the BCP if and only if $X$ has the BCP. They also proved that if $E$ is a Banach space with an $1$-unconditional basis $(e_{n})$, then the Banach space $X$ has the UBCP if and only if $E(X)$ has the UBCP, where $E(X)$ is the Banach space of sequences $(x_{n})\subseteq X$ with $\sum_{n}\|x_{n}\|e_{n}$ converging in $E$ and $\|(x_{n})\| =\|\sum_{n}\|x_{n}\|e_{n}\|$. Recently, Huang et al. \cite{HKL2023} characterized non-commutative symmetric spaces having the BCP, which provides a number of new examples of non-separable (commutative and non-commutative) Banach spaces having the BCP. They also showed that a von Neumann algebra has the BCP (indeed, the UBCP) if and only if it is atomic and can be represented on a separable Hilbert space. In \cite{LLLZ2022}, Liu et al. proved that $B(L^{1}[0, 1])$ fails the BCP. However, the following question is still open.

\begin{qu}\label{question3}
Does $B(L^{p}[0, 1])$ have the BCP for $1 <p <\infty$?
\end{qu}

Another main result of this paper answers Question \ref{question3} partially.
\begin{thm}
  Let $3/2 < p < 3$, then $B(L^p[0,1])$ has the UBCP.
\end{thm}

This paper is organized as follows: In Section 2, for convenience, we give a new proof for the result $\ell^{\infty}/ c_{0}$ fails the BCP by constructing recursively. We focus on the Banach space $X$ with a shrinking $1$-unconditional basis. We show that the norm of any operators in the quotient Banach algebra $B(X)/ K(X)$ can be approximated by the norm of an operator sequence in $B(X)$. Then by constructing an equivalent norm on $B(X)$, we present a characterization for the BCP of the renormed space $X_\alpha=(B(X),\|\cdot\|_\alpha)$.  We prove that $B(X)/ K(X)$ does not have the BCP, which implies that the Calkin algebra $B(H)/ K(H)$ where $H$ is an infinite-dimensional separable Hilbert space, $B(\ell^p)/K(\ell^p)$ ($1 \leq p <\infty$) and $B(c_0)/K(c_0)$ all fail the BCP. In Section 3, we prove that $B(L^p[0,1])$ has the UBCP for $3/2< p < 3$.

The following is a list of notations that will be used in this article.

 \begin{itemize}
    \item$\N_+$ --- the set of positive integers.
    \item$\N_{\geq m}$ --- the set of positive integers which are greater than or equal to positive integer $m$.
    \item$\N_+^2$ --- the set of positive integer pairs.
    \item$\operatorname{span} \{E\} $ --- the linear space spanned by the set $E$.
    \item $a \otimes b$ --- the rank one operator defined by $a\otimes b (x) = a(x)b$ for all $x,b \in X$ and $a \in X^*$.
    \item $id_X$--- the identity operator on Banach space $X$.
\end{itemize}


\section{The ball-covering property for $B(X)/ K(X)$ and renorming}

We need the following lemma first.
\begin{lem}\label{lem 1}
  There is a map $\pi:\N_+ \to \N_+$ such that for all $n,m \in \N_+$, we have $\pi^{-1}(n) \cap \N_{\geq m} \neq \emptyset.$
\end{lem}
\begin{proof}
  Since $\N_+^2$ is countable, there is a bijection $$\phi: \N_+ \to \N_+^2, \quad n \mapsto (\phi_1(n), \phi_2(n)),$$ where $\phi_i, i=1, 2$ is a map from $\N_+$ to $\N_+$, then $\pi= \phi_1: \N_+ \to \N_+$ is the desired map.
\end{proof}

The quotient algebra $\ell^{\infty}/ c_{0}$ is a commutative analog of the Calkin algebra and fails the BCP. We give a new proof for the following theorem which is different from the original proof in \cite{CCL2008}.

\begin{thm}
$\ell^{\infty}/ c_{0}$ fails the BCP.
\end{thm}
\begin{proof}
For any $\tilde{x}=\{\tilde{x}_{n}\}_{n=1}^{\infty}\in \ell^{\infty}/ c_{0}$, we have $\|\tilde{x}\|=\overline{\lim}_{n\rightarrow\infty}|\tilde{x}_{n}|$.
Suppose that $\ell^{\infty}/ c_{0}$ has the BCP, then there exists a sequence $\{B(\tilde{x}_{i}, r_{i})\}_{i=1}^{\infty}$ of open balls with $\|\tilde{x}_{i}\|\geq r_{i}>0$ for all $i\in \mathbb{N}$ such that the unit sphere $S$ of $\ell^{\infty}/ c_{0}$ is contained in $\bigcup_{i=1}^{\infty}B(\tilde{x}_{i},r_{i})$.
For $\tilde{x}_{i}\in \ell^{\infty}/ c_{0}$, there exists a subsequence $\{\tilde{x}_{i,n_{ik}}\}_{k=1}^{\infty}$ such that $$\lim_{k\rightarrow\infty}|\tilde{x}_{i,n_{ik}}|=\|\tilde{x}_{i}\|.$$

Then we will construct the outside sequence recursively. Let $\pi$ be the map of Lemma \ref{lem 1}.
First let $k=1$, $\lambda(1)=n_{11}$ and $$\tilde{x}_{0,\lambda(1)}=-\displaystyle\frac{\tilde{x}_{1,n_{11}}}{|\tilde{x}_{1,n_{11}}|}.$$ Suppose that for all $k\leq N$, $\tilde{x}_{0,\lambda(k)}$ and $\lambda(k)$ have already been constructed, then when $k=N+1$, let $\lambda(N+1)$ be the smallest integer of $\{n_{\pi(N+1),s}\}_{s=1}^{\infty}$ satisfying \[n_{\pi(N+1),s}>\lambda(N)\]
and let
$$\tilde{x}_{0,\lambda(N+1)}=-\displaystyle\frac{\tilde{x}_{\pi(N+1),\lambda(N+1)}}
{\left|\tilde{x}_{\pi(N+1),\lambda(N+1)}\right|}.$$
Finally, for any $j\neq \lambda(k)$, let $\tilde{x}_{0,j}=0$.
Define $\tilde{x}_{0} :=\{\tilde{x}_{0,j}\}_{j=1}^{\infty}$, then we have $\|\tilde{x}_{0}\|=1$.
For all $i,m\in\mathbb{N}$, since $\pi^{-1}(i) \cap \N_{\geq m} \neq \emptyset$, there is a subsequence $\{\pi(k_s)\}_{s=1}^\infty$ such that $\pi(k_s)$ is constant $i$.
So the subsequence $\{\tilde{x}_{0,\lambda(k_s)}\}_{s=1}^\infty$ satisfies 
\[ \left|\tilde{x}_{0,\lambda(k_s)}-\tilde{x}_{i,\lambda(k_s)}\right|=1+|\tilde{x}_{i,\lambda(k_s)}|.\]
Thus by the equal expression of the norm on $l^\infty / c_0$, we have
$$\|\tilde{x}_{0}-\tilde{x}_{i}\|\geq 1+\|\tilde{x}_{i}\|> r_{i}.$$
This implies that $\tilde{x}_{0}\notin\bigcup_{i=1}^{\infty}B(\tilde{x}_{i},r_{i})$, and it is a contradiction. So $\ell^{\infty}/ c_{0}$ fails the BCP.
\end{proof}

Let $(e_n)_{n=1}^\infty$ be a basis for a Banach space $X$. Recall that the basis $(e_{n})_{n=1}^{\infty}$ is {\it unconditional} if for each $x\in X$ the series $\sum_{n=1}^{\infty}e_{n}^{\ast}(x)e_{n}$ converges unconditionally and is {\it shrinking} if the coordinate functionals $(e_{n}^{\ast})_{n=1}^{\infty}$ are a basis for $X^\ast$ \cite{AK2011,LT1977,M1998}.

\begin{defn}[\cite{AK2011,LT1977,M1998}]
    Let $(e_n)_{n=1}^\infty$ be an unconditional basis of a Banach space $X$, then the unconditional basis constant $K_u$ is the smallest real number $K$ $(K\geq1)$ such  that for all $N \in \N_+$ and $|a_n| \leq |b_n|$ whenever $n=1,2,\cdots,N$, the following inequality holds
    \[ \norm{\sum_{n=1}^N a_n e_n} \leq K \norm{\sum_{n=1}^N b_n e_n}.\]
\end{defn}
If the unconditional basis constant of the unconditional basis $(e_n)_{n=1}^\infty$ is $1$, then it is said to be {\it $1$-unconditional}.


For all $T \in X_\alpha=(B(X),\|\cdot\|_\alpha)$, let $(e_n)_{n=1}^\infty$ be an $1$-unconditional basis for a Banach space $X$ with biorthogonal functionals $(e_n^*)_{n=1}^\infty$ (simplified as $\{(e_n,e_n^*)\}_{n=1}^\infty$), then
\begin{align*}
Tx  & =id_X T id_X (x)\\
    & =\sum_{n=1}^\infty e_n^*\left(T\left(\sum_{m=1}^\infty e_m^*(x)e_m\right)\right)e_n \\
    & =\sum_{n=1}^\infty\sum_{m=1}^\infty e_n^*T(e_m) e_m^*\otimes e_n (x).
\end{align*}
We denote \[T^{u,v}_{r,s}:=\sum_{n=u}^v \sum_{m=r}^s e_n^*T(e_m) e_m^* \otimes e_n\] for all $u,v,r,s \in \N_+\cup\{\infty\}$ with $u\leq v$ and $r \leq s$. If we denote the partial projection by $P_{r,s}$, that is,
\[P_{r,s} :=\sum_{n=r}^s e_n^*\otimes e_n,\]  since $\{(e_n,e_n^*)\}_{n=1}^\infty$ is $1$-unconditional, then for all $1\leq r \leq s \leq \infty$, we have
\[\norm{P_{r,s}} \leq 2.\]
So
\[T^{u,v}_{r,s}=P_{u,v} T P_{r,s}.\]
Clearly, $\TT{u}{v}{r}{s}$  has operator norm bound $4\|T\|$ and thus is well-defined.


The next lemma illustrates that if $X$ is a Banach space with an $1$-unconditional basis, then we can approximate the norm of any operator in the quotient Banach algebra $B(X)/ K(X)$ by the norm of an operator sequence in $B(X)$ combined with some partial projections.

\begin{lem} \label{lem 2}
  Let $(e_n)_{n=1}^\infty$ be an $1$-unconditional basis for a Banach space $X$ with biorthogonal functionals $(e_n^*)_{n=1}^\infty$  and $T \in B(X)/ K(X)$, then there are four sequences $\{w_i\}_{i=1}^\infty$, $\{W_i\}_{i=1}^\infty$, $\{v_i\}_{i=1}^\infty$ and $\{V_i\}_{i=1}^\infty$ such that \[\norm{T}_{B(X)/ K(X)}=\lim_i \norm{\TT{w_i}{W_i}{v_i}{V_i}}_{B(X)}.\]
\end{lem}
\begin{proof}
  Obviously, for all $x \in X$, we have $\TT{1}{K}{1}{\infty}=\sum_{n=1}^K e^*_nT \otimes e_n$. Since $\{(e_n,e_n^*)\}_{n=1}^\infty$ is $1$-unconditional, for all $x \in X$,
  we have
  \[\norm{\sum_{n=1}^K e^*_n T \otimes e_n (x)} \leq \norm{\sum_{n=1}^{K+1} e^*_n T \otimes e_n (x)} \leq \|T x\|.\]
  Thus
  \[\norm{\sum_{n=1}^K e^*_nT \otimes e_n}_{B(X)} \leq \norm{\sum_{n=1}^{K+1} e^*_nT \otimes e_n}_{B(X)} \leq \|T\|.\]
  Since the monotone bounded series of real numbers must have limit and clearly the limit of $\left\{\norm{\sum_{n=1}^{K} e^*_nT \otimes e_n}\right\}_{K=1}^\infty$ is $\|T\|$, there exists $K_1$ such that
  \[\left|\norm{\TT{1}{K_1}{1}{\infty}}-\norm{T}\right|<2^{-1}.\]
  Note that $$\left(\TT{K_1+1}{\infty}{1}{\infty}\right)^{K_1+1,s}_{1,\infty}=\TT{K_1+1}{s}{1}{\infty}.$$
  Similarly, for each $i$, there is $K_{i+1}$ such that
  \[\left|\norm{\TT{K_i+1}{K_{i+1}}{1}{\infty}}-\norm{\TT{K_i+1}{\infty}{1}{\infty}}\right|<2^{-i-1}.\]
  Thus we get a sequence $\{K_i\}_{i=1}^\infty$.
  Since $X$ has Schauder basis, the finite rank operator space $F(X)$ is operator norm dense in compact operator space $K(X)$ by the approximation $P_{1,n}C \to C$ $(n \to \infty) $ for any compact operator $C$.
  So \[\norm{T}_{B(X)/ K(X)}=\lim_i \norm{\TT{K_i}{\infty}{1}{\infty}}_{B(X)}.\]
  Take $m_i=K_i+1$ and $M_i=K_{i+1}$, then \[\norm{T}_{B(X)/ K(X)}=\lim_i \norm{\TT{m_i}{M_i}{1}{\infty}}_{B(X)}.\]

  Next, we will construct $\{v_i\}_{i=1}^\infty$ and $\{V_i\}_{i=1}^\infty$. By the same proof as the first half part, for all $i \in \N_+$, there exists a large enough $U_i>U_{i-1}$ (let $U_0=0$) such that
  \[\left| \norm{\TT{m_i}{M_i}{1}{\infty}}_{B(X)} -\norm{\TT{m_i}{M_i}{1}{U_i}}_{B(X)}\right| \leq 2^{-i}.\]
  We then check that for all $N \in \N_+$ and $\varepsilon>0$, there exist only finite many $\TT{m_i}{M_i}{1}{N}$ with $\norm{\TT{m_i}{M_i}{1}{N}} > \varepsilon$.
  Suppose not, let $\norm{\TT{m_{n_i}}{M_{n_i}}{1}{N}}> \varepsilon$ for all $i$, then there exists $x_i\in X$ with $\|x_i\|=1$ such that \[\norm{\TT{m_{n_i}}{M_{n_i}}{1}{N}x_i }>2^{-1}\varepsilon.\]
  Since \[\TT{m_{n_i}}{M_{n_i}}{1}{N}= \TT{m_{n_i}}{M_{n_i}}{1}{N}|_{\operatorname{span}\{e_i\}_{i=1}^N}\]  and the unit sphere of $\operatorname{span}\{e_i\}_{i=1}^N$ is sequentially compact, $\{x_i\}_{i=1}^\infty$  has a convergent subsequence. Without loss of generality,  by passing to a subsequence,  we can assume that for a fixed $x_j$ and all $\TT{m_{n_i}}{M_{n_i}}{1}{N}$  we have
  \[\norm{\TT{m_{n_i}}{M_{n_i}}{1}{N}x_j}>4^{-1}\varepsilon.\]
  We claim that $\norm{\sum_{i=1}^\infty\TT{m_{n_{r_i}}}{M_{n_{r_i}}}{1}{N}x_j}=\infty$ for some subsequence $\{n_{r_i}\}_{i=1}^\infty$.
  If not, then \[E:=\left\{\pm \sum_{i=1}^\infty\TT{m_{n_{r_i}}}{M_{n_{r_i}}}{1}{N}x_j: \{n_{r_i}\}_{i=1}^\infty \subseteq \{n_i\}_{i=1}^\infty \right\}\] will be a well-defined subset of $X$. Clearly, the cardinal number of $E$ is $2^{\aleph_0}$ which contradicts with that $X$ has a Schauder basis.
  Thus we have
  \[\infty=\norm{\sum_{i=1}^\infty\TT{m_{n_{r_i}}}{M_{n_{r_i}}}{1}{N}x_j} \leq \norm{\sum_{i=1}^\infty\TT{m_{n_{r_i}}}{M_{n_{r_i}}}{1}{N}} \leq \|T\|,\]
which is a contradiction.

  By the finiteness of $\TT{m_i}{M_i}{1}{N}$ with \[\norm{\TT{m_i}{M_i}{1}{N}} > 2^{-2}\] for $N=U_2$, there is a large enough $n_2>n_1+1$ where $n_1=1$ such that
  \begin{align*}
    \quad\left|\norm{\TT{m_{n_2}}{M_{n_2}}{1}{U_{n_2}}}_{B(X)}-\norm{\TT{m_{n_2}}{M_{n_2}}{U_2+1}{U_{n_2}}}_{B(X)}\right|\leq \norm{\TT{m_{n_2}}{M_{n_2}}{1}{U_{2}}}_{B(X)}\leq 2^{-2}.
  \end{align*}
  If for all $r<R$, $n_r$ has been chosen. Then by the finiteness of  $\TT{m_i}{M_i}{1}{N}$ with \[\norm{\TT{m_i}{M_i}{1}{N}} > 2^{-R}\] for $N=U_{n_{R-1}}$, there is a large enough $n_{R}>n_{R-1}+1$ such that
  \begin{align*}
    \quad\left|\norm{\TT{m_{n_R}}{M_{n_R}}{1}{U_{n_R}}}_{B(X)}-\norm{\TT{m_{n_R}}{M_{n_R}}{U_{n_{R-1}+1}}{U_{n_R}}}_{B(X)}\right|\leq \norm{\TT{m_{n_R}}{M_{n_R}}{1}{U_{n_{R-1}}}}_{B(X)}\leq 2^{-R}.
  \end{align*}
  Thus we get a strictly monotone increasing sequence $\{n_j\}_{j=1}^\infty$. For all $j\in \N_+$, we define
  \[ w_j = m_{n_j} , W_j=M_{n_j}, v_j=U_{n_{j-1}}+1\ (v_1=1) \text{ and } V_j=U_{n_j}. \]
  Then for all $n \in \N_+$,
   \[\left| \norm{\TT{w_i}{W_i}{1}{\infty}}_{B(X)} -\norm{\TT{w_i}{W_i}{v_i}{V_i}}_{B(X)}\right| \leq 2^{-i}+2^{-i} = 2^{-i+1}.\]
   Therefore we have \[\norm{T}_{B(X)/ K(X)}=\lim_i \norm{\TT{w_i}{W_i}{1}{\infty}}_{B(X)}=\lim_i \norm{\TT{w_i}{W_i}{v_i}{V_i}}_{B(X)}.\]
\end{proof}

Then we consider the BCP of the renormed space $X_\alpha=(B(X),\|\cdot\|_{\alpha})$.

\begin{thm}\label{main1}
 Let $(e_n)_{n=1}^\infty$ be an $1$-unconditional basis for a Banach space $X$ with biorthogonal functionals $(e_n^*)_{n=1}^\infty$ and $0 \leq \alpha \leq 1/2$, then the renormed space $X_{\alpha}=(B(X), \|\cdot\|_{\alpha})$ fails the BCP.
\end{thm}
\begin{proof}
  For any $0 \leq \alpha \leq 1/2$, suppose the contrary holds, then there is a sequence $\{B(T(n),r_n)\}_{n=1}^\infty$ such that for all $n \in \N_+$ we have $r_n <\norm{T(n)}_\alpha$ and the unit sphere of $X_\alpha$ is contained in $\bigcup_{n=1}^\infty B(T(n),r_n)$.
  For all $T(n)$, by Lemma \ref{lem 2}, we know that there are $\{m_{n,i}\}_{i=1}^\infty$, $\{M_{n,i}\}_{i=1}^\infty$, $\{u_{n,i}\}_{i=1}^\infty$ and $\{U_{n,i}\}_{i=1}^\infty$ such that  \[\norm{T(n)}_{B(X)/ K(X)}=\lim_i \norm{\TTn{m_{n,i}}{M_{n,i}}{u_{n,i}}{U_{n,i}}{n}}_{B(X)}.\]

  Let $\pi$ be the map in Lemma \ref{lem 1}. Then we will construct the outside point. Firstly, for $k=1$, let $\lambda(1)=1$ and
  \[ T_1 = -\frac{ \TTn{m_{\pi(1),1}}{M_{\pi(1),1}}{u_{\pi(1),1}}{U_{\pi(1),1}}{\pi(1)} }{ \norm{ \TTn{m_{\pi(1),1}}{M_{\pi(1),1}}{u_{\pi(1),1}}{U_{\pi(1),1}}{\pi(1)}}_{B(X)} } . \]
  Suppose that for all $k\leq N$, $T_k$ has been constructed. For $k=N+1$, let $\lambda(N+1)$ be the smallest integer such  that \[U_{\pi(N),\lambda(N)}<u_{\pi(N+1),\lambda(N+1)}\quad  \text{ and } \quad M_{\pi(N),\lambda(N)}<m_{\pi(N+1),\lambda(N+1)}.\]
  Then we define
  \[T_{N+1}=T_N-\frac{ \TTn{m_{\pi(N+1),\lambda(N+1)}}{M_{\pi(N+1),\lambda(N+1)}}{u_{\pi(N+1),\lambda(N+1)}}{U_{\pi(N+1),\lambda(N+1)}}{\pi(N+1)} }{ \norm{\TTn{m_{\pi(N+1),\lambda(N+1)}}{M_{\pi(N+1),\lambda(N+1)}}{u_{\pi(N+1),\lambda(N+1)}}{U_{\pi(N+1),\lambda(N+1)}}{\pi(N+1)}}_{B(X)} } . \]
  Thus we get an operator sequence $\{T_n\}_{n=1}^\infty$ which satisfies $\|T_n\|_{B(X)}=1$ and pointwisely converges to \[T_0=-\sum_{n=1}^{\infty}\frac{ \TTn{m_{\pi(n),\lambda(n)}}{M_{\pi(n),\lambda(n)}}{u_{\pi(n),\lambda(n)}}{U_{\pi(n),\lambda(n)}}{\pi(n)} }{ \norm{ \TTn{m_{\pi(n),\lambda(n)}}{M_{\pi(n),\lambda(n)}}{u_{\pi(n),\lambda(n)}}{U_{\pi(n),\lambda(n)}}{\pi(n)} }_{B(X)} }.\]
  Since $X$ is a Banach space with an $1$-unconditional basis, we have
  \[\norm{T_0}_{B(X)}=1 \quad\text{ and  }\quad \norm{T_0}_{B(X)/K(X)}=1.\] Therefore $\norm{T_0}_\alpha=1$.
  For all $T(n)$, by Lemma \ref{lem 1}, there is a subsequence $\{\pi(k_s)\}_{s=1}^\infty$ such that  $\pi(k_s)$ is constant $n$. Since  $\TTn{m_{n,\lambda(k_s)}}{M_{n,\lambda(k_s)}}{u_{n,\lambda(k_s)}}{U_{n,\lambda(k_s)}}{n}$ is actually an operator from a finite-dimensional space to a finite-dimensional space, there is $x_s$ with $\|x_s\|=1$ which attains its operator norm.
  Thus we have
  \begin{align*}
    \norm{T_0-T(n)}_{B(X)/K(X)} &\geq  \lim_s \norm{\left(T_0-T(n) \right)x_s} \\
    &\geq \lim_s \norm{\left(T_0-\TTn{m_{n,\lambda(k_s)}}{M_{n,\lambda(k_s)}}{u_{n,\lambda(k_s)}}{U_{n,\lambda(k_s)}}{n}\right)x_s}\\
    &=1+\norm{T(n)}_{B(X)/K(X)}.
  \end{align*}
Therefore
  \begin{align*}
    \norm{T(n)-T_0}_\alpha&= \alpha\norm{T(n)-T_0}_{B(X)}+(1-\alpha)\norm{T(n)-T_0}_{B(X)/K(X)} \\
    & \geq \alpha\norm{T(n)-T_0}_{B(X)} +(1 -\alpha) \left(1+\norm{T(n)}_{B(X)/K(X)}\right)\\
    &\geq \alpha\left(\|T(n)\|_{B(X)}-1\right)+(1 -\alpha) \left(1+\norm{T(n)}_{B(X)/K(X)}\right)\\
    &=\norm{T(n)}_\alpha + 1-2\alpha \\
    &\geq \norm{T(n)}_\alpha.
  \end{align*}
  This inequality shows that  $$T_0 \notin \bigcup_{n=1}^\infty B(T(n),r_n).$$
  Thus $X_{\alpha}=(B(X), \|\cdot\|_{\alpha})$ fails the BCP.
\end{proof}

For all $1 \leq p<\infty$, the canonical Schauder basis of $\ell_p$ and $c_0$ is 1-unconditional. Particularly, when $p=2$ we know $B(\ell^2)/K(\ell^2)$ is the Calkin algebra on the separable Hilbert space. Thus we have the following corollaries.


\begin{cor}
   $B(\ell^p)/K(\ell^p)$ and $B(c_0)/K(c_0)$ fail the BCP.
\end{cor}

\begin{cor}
 Let $H$ be an infinite-dimensional separable Hilbert space, then the Calkin algebra $B(H)/K(H)$ fails the BCP.
\end{cor}





Next we will consider when the renormed space $X_\alpha$ has the BCP. 

\begin{thm} \label{BCP1}
  Let $(e_n)_{n=1}^\infty$ be an $1$-unconditional basis for a Banach space $X$ with biorthogonal functionals $(e_n^*)_{n=1}^\infty$ and $X^*$ be separable. If $1/2<\alpha \leq 1$, then the renormed space $X_{\alpha}=(B(X), \|\cdot\|_{\alpha})$ has the BCP.
\end{thm}
\begin{proof}
  We first show that if $\alpha=1$ then $B(X)=X_1$ has the BCP and all BCP points can be chosen in $K(X)$.
  Since $X^*$ is separable, let $\mathcal{A}=\{x_n^*\}_{n=1}^\infty$ be the countable dense subset of the unit ball of $X^*$. For all $T \in B(X)$ with $\|T\|=1$ and $x \in X$, we have
  \[Tx=\sum_{n=1}^{\infty}e_n^*T\otimes e_n (x).\]
  Since $\{(e_n,e_n^*)\}_{n=1}^\infty$ is 1-unconditional, for all $K=1,2,\cdots$ and for all $x \in X$,
  we have
  \[\norm{\sum_{n=1}^K e^*_n T \otimes e_n (x)} \leq \norm{\sum_{n=1}^{K+1} e^*_n T \otimes e_n (x)} \leq \|T x\|.\]
  Therefore
  \[\norm{\sum_{n=1}^{K} e_n^*T\otimes e_n} \leq \norm{\sum_{n=1}^{K+1} e_n^*T\otimes e_n} \leq \|T\|=1.\]
  Note that the monotone bounded series $\left\{\norm{\sum_{n=1}^{K} e^*_nT \otimes e_n}\right\}_{K=1}^\infty$ must have limit and the limit is precisely 1.
  For all $0<\varepsilon_3<1$, let $0<\varepsilon_1<\min(\varepsilon_3/8,1/4)$,  then there is a large enough $K$ such that
  \[\|T\|-\varepsilon_1=1-\varepsilon_1 \leq \norm{\sum_{n=1}^{K}e_n^*T\otimes e_n} \leq 1=\|T\|.\]
  For all $n=1,2,\cdots,K$ and $0<\varepsilon_2<\min(1/4K,\varepsilon_3/16K)$, there exists $x^*_{m_n} \in \{x_i^*\}_{i=1}^\infty$, $m_n\in \N_+$ such that
  $$\norm{x^*_{m_n}-e_n^* T} \leq \varepsilon_2.$$
  Thus by triangular inequality, we have \[\norm{\sum_{n=1}^{K}e_n^*T\otimes e_n -\sum_{n=1}^{K}x^*_{m_n}\otimes e_n}\leq K\varepsilon_2\] and
  \[\|T\|-\varepsilon_1-K\varepsilon_2=1-\varepsilon_1-K\varepsilon_2 \leq\norm{\sum_{n=1}^{K} x_{m_n}^*\otimes e_n} \leq 1+K\varepsilon_2=\|T\|+K\varepsilon_2.\] 
  
  Now we show that the countable set 
  \[\left\{\frac{2\sum_{i=1}^{K} f_i\otimes e_i}{\norm{\sum_{i=1}^{K} f_{i}\otimes e_i}}: K \in \N_+, f_i\in\mathcal{A}\ (1\leq i\leq K)\right\}\] is a set of BCP points for $B(X)$.
  
  Since $\{(e_n,e_n^*)\}_{n=1}^\infty$ is 1-unconditional, we have 
  \begin{align}
    &\quad\norm{T - \frac{2\sum_{n=1}^{K} x_{m_n}^*\otimes e_n}{\norm{\sum_{n=1}^{K} x_{m_n}^*\otimes e_n}}}\notag\\
    &=\norm{\sum_{n=1}^K \left(e_n^* T-\frac{2 x_{m_n}^*}{\norm{\sum_{n=1}^{K} x_{m_n}^*\otimes e_n}} \right)\otimes e_n +\sum_{n=K+1}^\infty e_n^* T\otimes e_n }\notag\\
    &\leq \norm{\sum_{n=1}^K \left(\frac{2 e_n^* T}{\norm{\sum_{n=1}^{K} x_{m_n}^*\otimes e_n}} \right)\otimes e_n -\sum_{n=1}^{K}\left(\frac{2 x_{m_n}^*}{\norm{\sum_{n=1}^{K} x_{m_n}^*\otimes e_n}} \right)\otimes e_n }\notag\\
    &\quad +\norm{\left(1-\frac{2}{\norm{\sum_{n=1}^{K} x_{m_n}^*\otimes e_n}} \right)\sum_{n=1}^K e_n^* T\otimes e_n+\sum_{n=K+1}^\infty e_n^* T\otimes e_n } \notag\\
    & \leq \frac{2K\varepsilon_2}{1-\varepsilon_1-K\varepsilon_2}+\max\left(\left|1-\frac{2}{\norm{\sum_{n=1}^{K} x_{m_n}^*\otimes e_n}} \right| ,1\right) \notag \\
    &\leq \frac{2K\varepsilon_2}{1-\varepsilon_1-K\varepsilon_2}+\frac{2}{1-\varepsilon_1-K\varepsilon_2}-1\notag\\
    &=1+\frac{2\varepsilon_1+4K\varepsilon_2}{1-\varepsilon_1-K\varepsilon_2}\notag\\
    &\leq 1+2\left(\frac{\varepsilon_3}{4}+\frac{\varepsilon_3}{4}\right)\notag\\
    &=1+\varepsilon_3 \notag\\
    &<2.\notag
  \end{align}

  Next we assume $1/2<\alpha < 1$. For all $T \in X_\alpha$ with $\|T\|_\alpha=1$, we have $\|T\|_{B(X)}\geq 1$ and $\|T\|_{B(X)/K(X)} \leq 1$. We will show that the countable set \[\left\{\frac{2\sum_{i=1}^{K} f_i\otimes e_i}{\norm{\sum_{i=1}^{K} f_i\otimes e_i}_{B(X)}}: K\in \N_+, f_i\in\mathcal{A}\ (1\leq i\leq K)\right\}\] is a set of BCP points for $X_\alpha$ $(1/2<\alpha\leq 1)$ and the radius of balls is $2\alpha$.
  
  Obviously, $$ \norm{2\left(\norm{\sum_{n=1}^{K} x_{m_n}^*\otimes e_n}_{B(X)}\right)^{-1}\sum_{n=1}^{K} x_{m_n}^*\otimes e_n }_\alpha =2 \alpha >1.$$ For all \[0<\varepsilon_3<2\alpha-1,\] let $\varepsilon_1, \varepsilon_2, K$ and $\{x^*_{m_n}\}_{n=1}^\infty$ be chosen the same as before. Since $\{(e_n,e_n^*)\}_{n=1}^\infty$ is 1-unconditional, we have
  \begin{align*}
    &\quad\norm{T - \frac{2\sum_{n=1}^{K} x_{m_n}^*\otimes e_n}{\norm{\sum_{n=1}^{K} x_{m_n}^*\otimes e_n}_{B(X)}}}_\alpha\\
    &=\alpha\norm{T - \frac{2\sum_{n=1}^{K} x_{m_n}^*\otimes e_n}{\norm{\sum_{n=1}^{K} x_{m_n}^*\otimes e_n}_{B(X)}}}_{B(X)}\\
    &\quad +(1-\alpha)\norm{T - \frac{2\sum_{n=1}^{K} x_{m_n}^*\otimes e_n}{\norm{\sum_{n=1}^{K} x_{m_n}^*\otimes e_n}_{B(X)}}}_{B(X)/ K(X)}\\
    &=\alpha\norm{\sum_{n=1}^K \left(e_n^* T-\frac{2 x_{m_n}^*}{\norm{\sum_{n=1}^{K} x_{m_n}^*\otimes e_n}_{B(X)}} \right)\otimes e_n +\sum_{n=K+1}^\infty e_n^* T\otimes e_n }_{B(X)}\\
    &\quad +(1-\alpha)\norm{T - \frac{2\sum_{n=1}^{K} x_{m_n}^*\otimes e_n}{\norm{\sum_{n=1}^{K} x_{m_n}^*\otimes e_n}_{B(X)}}}_{B(X)/ K(X)}\\
    &\leq \alpha\norm{\sum_{n=1}^K \left(\frac{2 e_n^* T}{\norm{\sum_{n=1}^{K} x_{m_n}^*\otimes e_n}_{B(X)}}-\frac{2 x_{m_n}^*}{\norm{\sum_{n=1}^{K} x_{m_n}^*\otimes e_n}_{B(X)}} \right)\otimes e_n }_{B(X)}\\
    &\quad +\alpha\norm{\left(1-\frac{2}{\norm{\sum_{n=1}^{K} x_{m_n}^*\otimes e_n}_{B(X)}} \right)\sum_{n=1}^K e_n^* T\otimes e_n+\sum_{n=K+1}^\infty e_n^* T\otimes e_n }_{B(X)}\\
    &\quad +(1-\alpha)\norm{T - \frac{2\sum_{n=1}^{K} x_{m_n}^*\otimes e_n}{\norm{\sum_{n=1}^{K} x_{m_n}^*\otimes e_n}_{B(X)}}}_{B(X)/K(X)}\\
    &\leq \frac{2K\alpha\varepsilon_2}{\norm{T}_{B(X)}-\varepsilon_1-K\varepsilon_2}+\alpha\left|1-\frac{2}{1-\varepsilon_1-K\varepsilon_2}\right|\|T\|_{B(X)}+(1-\alpha)\norm{T}_{B(X)/K(X)}\\
    &\leq 1 +  \frac{2K\alpha\varepsilon_2}{1-\varepsilon_1-K\varepsilon_2}+\alpha\|T\|_{B(X)}\left(\frac{2}{1-\varepsilon_1-K\varepsilon_2}-2\right)\\
    &\leq 1+\frac{2\varepsilon_1+2(1+\alpha) K \varepsilon_2}{1-\varepsilon_1-K\varepsilon_2}\\
    &\leq 1+\varepsilon_3\\
    &<2\alpha\\
    &=\norm{\frac{2\sum_{n=1}^{K} x_{m_n}^*\otimes e_n}{\norm{\sum_{n=1}^{K} x_{m_n}^*\otimes e_n}_{B(X)}}}_{\alpha}.
  \end{align*}
  This finishes the proof.
\end{proof}


Combining Theorem \ref{main1} with Theorem \ref{BCP1}, we obtain the main result of this section.
\begin{thm}
  Let $X$ be a Banach space with a shrinking $1$-unconditional basis $(e_n)_{n=1}^\infty$. Then the renormed space $X_{\alpha}=(B(X), \|\cdot\|_{\alpha})$ has the BCP if and only if $\alpha > 1/2$.
\end{thm}
\begin{proof}
Necessity. Fix any $0\leq \alpha\leq1/2$, by Theorem \ref{main1}, the renormed space $X_\alpha$ does not have the BCP.

Sufficiency. Assume $1/2<\alpha\leq1$. Since $(e_n)_{n=1}^\infty$ is shrinking, the coordinate functionals $(e_{n}^{\ast})_{n=1}^{\infty}$ are a basis for $X^\ast$, this implies that $X^\ast$ is separable. Then by Theorem \ref{BCP1}, the renormed space $X_\alpha$ has the BCP.
\end{proof}

For all $1 < p<\infty$, the canonical Schauder basis of $\ell_p$ is 1-unconditional and its dual space $l^q$ where $p^{-1}+q^{-1}=1$ is separable. The canonical Schauder basis of $c_0$ is also 1-unconditional and its dual space $l^1$ is also separable. Thus we have the following corollary.

\begin{cor}
  $\left(B(\ell^p), \|\cdot\|_\alpha\right)$ $(1<p<\infty)$ and $\left(B(c_0), \|\cdot\|_\alpha\right)$ have the BCP if and only if $\alpha > 1/2$.
\end{cor}

Since the Haar basis of $L^2[0,1]$ is 1-unconditional and its dual space is separable, we have the following result.

\begin{cor}
  $\left(B(L^2[0,1]), \|\cdot\|_\alpha\right)$ has the BCP if and only if $\alpha > 1/2$.
\end{cor}




\begin{rem}
Notice that the precondition in above results can not extend to the Banach spaces with monotone basis and considering the Banach spaces with $1$-unconditional basis is essential. In fact, there exist plenty of Banach spaces $X$ with a Schauder basis (and hence, after renorming, a monotone Schauder basis \cite{LT1977}) such that $B(X)/K(X)$ is separable and thus satisfies the BCP under any equivalent norm. The first example is the Argyros-Haydon space \cite{AH2011} where $B(X)/K(X)$ is one-dimensional. Other examples are given by M. Tarbard in \cite{T2012}. 
\end{rem}

\section{The ball-covering property of $B(L^p[0,1])$}


As mentioned before, $B(L^1[0,1])$ fails the BCP. In order to determine whether other $B(L^p[0,1])$ for $1<p<\infty$ has the BCP, we need the following lemma which shows the unconditional constant of the Haar basis of $L^p[0,1]$.

\begin{lem}[\cite{AK2011,C1992,M1998}] \label{unc}
  Let $1<p<\infty$ and $p^{-1}+q^{-1}=1$, then the Haar basis in $L^p[0,1]$ is an unconditional basis and the unconditional constant $K_u(p)$ is accurately $\max(p,q)-1$.
\end{lem}

\begin{lem}[\cite{AK2011,C1992,M1998}]\label{monotone}
  Haar basis is a monotone basis of $L^p[0,1]$ for all $1 \leq p <\infty$.
\end{lem}

\begin{thm} \label{LpBCP}
  Let $3/2 < p < 3$, then $B(L^p[0,1])$ has the UBCP.
\end{thm}
\begin{proof}
  Let $(e_n)_{n=1}^\infty$ be the Haar basis of $L^p[0,1]$ with biorthogonal functionals $(e_n^*)_{n=1}^\infty$ (simplified as $\{(e_n,e_n^*)\}_{n=1}^\infty$). Since $(L^p[0,1])^*=L^q[0,1]$ where $p^{-1}+q^{-1}=1$ is separable, let $\mathcal{A}=\{x_n^*\}_{n=1}^\infty$ be the countable dense subset of the unit ball of $L^q[0,1]$. For all $T \in B(L^p[0,1])$ with $\|T\|=1$ and $x \in X$, we have
  \[Tx=\sum_{n=1}^{\infty}e_n^*T\otimes e_n (x).\]
  By Lemma \ref{monotone}, $\{(e_n,e_n^*)\}_{n=1}^\infty$ is monotone, then for all $K=1,2,\cdots$ and for all $x \in X$,
  we have
  \[\norm{\sum_{n=1}^K e^*_n T \otimes e_n (x)} \leq \norm{\sum_{n=1}^{K+1} e^*_n T \otimes e_n (x)} \leq \|T x\|.\]
  Therefore
  \[\norm{\sum_{n=1}^{K} e_n^*T\otimes e_n} \leq \norm{\sum_{n=1}^{K+1} e_n^*T\otimes e_n} \leq \|T\|=1.\]
  Since the monotone bounded series $\left\{\norm{\sum_{n=1}^{K} e^*_nT \otimes e_n}\right\}_{K=1}^\infty$ must have limit and the limit is precisely 1,
  by Lemma \ref{unc}, the unconditional constant of $\{(e_n,e_n^*)\}_{n=1}^\infty$ is $2-\kappa$ for some $$\kappa=3-\max(p,q)>0.$$
  Then for all \[0<\varepsilon_3< \frac{2\kappa}{7-3\kappa},\] let $0<\varepsilon_1<\min(\varepsilon_3/8,1/4)$, then there is a large enough $K$ such that \[1-\varepsilon_1 \leq \norm{\sum_{n=1}^{K}e_n^*T\otimes e_n} \leq 1.\]
  For all $n=1,2,\cdots,K$ and $0<\varepsilon_2<\min(1/4K,\varepsilon_3/16K)$, there exists $x^*_{m_n} \in \{x_i^*\}_{i=1}^\infty$, $m_n\in\N_+$ such that
  $$\norm{x^*_{m_n}-e_n^* T} \leq \varepsilon_2.$$
  So by triangular inequality, we have \[\norm{\sum_{n=1}^{K}e_n^*T\otimes e_n -\sum_{n=1}^{K}x^*_{m_n}\otimes e_n}\leq K\varepsilon_2\] and
  \[1-\varepsilon_1-K\varepsilon_2 \leq\norm{\sum_{n=1}^{K} x_{m_n}^*\otimes e_n} \leq 1+K\varepsilon_2.\]
  We will show that the countable set
  \[\left\{\frac{2\sum_{i=1}^{K} f_i\otimes e_i}{\norm{\sum_{i=1}^{K} f_i\otimes e_i}}: K\in \N_+, f_i\in\mathcal{A}\ (1\leq i\leq K )\right\}\] is a set of UBCP points for $B(L^p[0,1])$ $(3/2 < p < 3)$ and the radius of balls is $2-\kappa/2$.
  
  Actually we have 
  \begin{align}
    &\quad\norm{T - \frac{2\sum_{n=1}^{K} x_{m_n}^*\otimes e_n}{\norm{\sum_{n=1}^{K} x_{m_n}^*\otimes e_n}}}\notag\\
    &=\norm{\sum_{n=1}^K \left(e_n^* T-\frac{2 x_{m_n}^*}{\norm{\sum_{n=1}^{K} x_{m_n}^*\otimes e_n}} \right)\otimes e_n +\sum_{n=K+1}^\infty e_n^* T\otimes e_n }\notag\\
    &\leq \norm{\sum_{n=1}^K \left(\frac{2 e_n^* T}{\norm{\sum_{n=1}^{K} x_{m_n}^*\otimes e_n}} \right)\otimes e_n -\sum_{n=1}^K\left(\frac{2 x_{m_n}^*}{\norm{\sum_{n=1}^{K} x_{m_n}^*\otimes e_n}} \right)\otimes e_n }\notag\\
    &\quad +\norm{\left(1-\frac{2}{\norm{\sum_{n=1}^{K} x_{m_n}^*\otimes e_n}} \right)\sum_{n=1}^K e_n^* T\otimes e_n+\sum_{n=K+1}^\infty e_n^* T\otimes e_n } \notag\\
    & \leq \frac{2K\varepsilon_2}{1-\varepsilon_1-K\varepsilon_2}+(2-\kappa)\max\left(\left|1-\frac{2}{\norm{\sum_{n=1}^{K} x_{m_n}^*\otimes e_n}} \right| ,1\right)\label{eq2}\\
    &\leq \frac{2K\varepsilon_2}{1-\varepsilon_1-K\varepsilon_2}+
    (2-\kappa)\left(\frac{2}{1-\varepsilon_1-K\varepsilon_2}-1\right)\notag\\
    &= 2- \kappa +(2-\kappa)\frac{2\varepsilon_1+2K\varepsilon_2}{1-\varepsilon_1-K\varepsilon_2}+
    \frac{2K\varepsilon_2}{1-\varepsilon_1-K\varepsilon_2}\notag\\
    &\leq 2-\kappa+ \frac{7-3\kappa}{4}\varepsilon_3\notag\\
    &< 2-\frac{\kappa}{2} \notag\\
    &= \norm{\frac{2\sum_{n=1}^{K} x_{m_n}^*\otimes e_n}{\norm{\sum_{n=1}^{K} x_{m_n}^*\otimes e_n}}}-\frac{\kappa}{2}\notag,
  \end{align}
  where inequality (\ref{eq2}) is obtained by unconditional constant $2-\kappa$. This finishes the proof.
\end{proof}

\begin{cor}
Let $X$ be a Banach space and $X^*$ be separable. Let $(e_n)_{n=1}^\infty$ be a basis for $X$ with biorthogonal functionals $(e_n^*)_{n=1}^\infty$ such that $\{(e_n,e_n^*)\}_{n=1}^\infty$ is both monotone and $(2-\varepsilon)$-unconditional for some $\varepsilon>0$. If $1-\varepsilon/2<\alpha\leq1$, then the renormed space $X_{\alpha}=(B(X),\|\cdot\|_\alpha)$ has the BCP.
\end{cor}
\begin{proof}
We first show that if $\alpha=1$ then $X_1=B(X)$ has the BCP. Since $X^*$ is separable, we can let $\mathcal{A}=\{x_n^*\}_{n=1}^\infty$ be the countable dense subset of the unit ball of $X^\ast$. Since $\{(e_n,e_n^*)\}_{n=1}^\infty$ is a basis of $X$, for all $T \in B(X)$ with $\|T\|=1$ and for all $x \in X$, we have
  \[Tx=\sum_{n=1}^{\infty}e_n^*T\otimes e_n (x).\]
Since the basis $\{(e_n,e_n^*)\}_{n=1}^\infty$ is monotone, then for all $K=1,2,\cdots$, we have
  \[\norm{\sum_{n=1}^{K} e_n^*T\otimes e_n} \leq \norm{\sum_{n=1}^{K+1} e_n^*T\otimes e_n} \leq \|T\|=1.\]
  Note that the monotone bounded series $\left\{\norm{\sum_{n=1}^{K} e^*_nT \otimes e_n}\right\}_{K=1}^\infty$ must have limit and the limit is precisely 1. Since the unconditional constant of $\{(e_n,e_n^*)\}_{n=1}^\infty$ is $2-\varepsilon$, then for all $0<\varepsilon_3< 2\varepsilon/ (7-3\varepsilon)$, let $0<\varepsilon_1<\min(\varepsilon_3/8,1/4)$, there is a large enough $K$ such that \[\|T\|-\varepsilon_1=1-\varepsilon_1 \leq \norm{\sum_{n=1}^{K}e_n^*T\otimes e_n} \leq 1=\|T\|.\]
  For all $n=1,2,\cdots,K$ and $0<\varepsilon_2<\min(1/4K,\varepsilon_3/16K)$, there exists $x^*_{m_n} \in \{x_i^*\}_{i=1}^\infty$, $m_n\in\N_+$ such that $\norm{x^*_{m_n}-e_n^* T} \leq \varepsilon_2$.
  Since $\{(e_n,e_n^*)\}_{n=1}^\infty$ is $(2-\varepsilon)$-unconditional for some $\varepsilon>0$, by the proof of  Theorem \ref{LpBCP}, we obtain that the countable set
  \[\left\{\frac{2\sum_{i=1}^{K} f_i\otimes e_i}{\norm{\sum_{i=1}^{K} f_i\otimes e_i}}: K\in \N_+, f_i\in\mathcal{A}\ (1\leq i\leq K)\right\}\] is a set of BCP points for $B(X)$ and the radius of balls is $2-\varepsilon/2$.

  Now we assume $1-\varepsilon/2<\alpha < 1$. For all $T \in X_\alpha$ with $\|T\|_\alpha=1$, we have $\|T\|_{B(X)}\geq 1$ and $\|T\|_{B(X)/K(X)} \leq 1$. For all \[0<\varepsilon_3<\frac{2\alpha-2+\varepsilon}{5},\] let $\varepsilon_1, \varepsilon_2, K$ and $x_{m_n}^*$ be chosen the same as before, then we have
  \begin{align*}
    &\quad\norm{T - \frac{2\sum_{n=1}^{K} x_{m_n}^*\otimes e_n}{\norm{\sum_{n=1}^{K} x_{m_n}^*\otimes e_n}_{B(X)}}}_\alpha\\
    &\leq \frac{2\alpha K\varepsilon_2}{1-\varepsilon_1-K\varepsilon_2}+\alpha(2-\varepsilon)
    \left(\frac{1+\varepsilon_1+K\varepsilon_2}{1-\varepsilon_1-K\varepsilon_2}\right)\norm{T}_{B(X)}+(1-\alpha)\norm{T}_{B(X)/K(X)}\\
    &=1+\frac{2\alpha K\varepsilon_2}{1-\varepsilon_1-K\varepsilon_2}+\alpha(2-\varepsilon)
    \left(\frac{1+\varepsilon_1+K\varepsilon_2}{1-\varepsilon_1-K\varepsilon_2}\right)\norm{T}_{B(X)}-\alpha\|T\|_{B(X)}\\
    &=1+\frac{2\alpha K\varepsilon_2}{1-\varepsilon_1-K\varepsilon_2}+\alpha(1-\varepsilon)\|T\|_{B(X)}+\alpha(2-\varepsilon)\frac{2\varepsilon_1+2K\varepsilon_2}{1-\varepsilon_1-K\varepsilon_{2}}\|T\|_{B(X)}\\
    &\leq 1+\frac{2\alpha K\varepsilon_2}{1-\varepsilon_1-K\varepsilon_2}+(1-\varepsilon)+(2-\varepsilon)\frac{2\varepsilon_1+2K\varepsilon_2}{1-\varepsilon_1-K\varepsilon_{2}}\\
    &\leq 1+(1-\varepsilon)+5\varepsilon_3\\
    &<2\alpha \\
    &= \norm{\frac{2\sum_{n=1}^{K} x_{m_n}^*\otimes e_n}{\norm{\sum_{n=1}^{K} x_{m_n}^*\otimes e_n}}_{B(X)}}_\alpha.
  \end{align*}
  This shows that the countable set \[\left\{\frac{2\sum_{i=1}^{K} f_i\otimes e_i}{\norm{\sum_{i=1}^{K} f_i\otimes e_i}_{B(X)}}: K\in \N_+, f_i\in \mathcal{A}\ (1\leq i\leq K)\right\}\] is a set of BCP points for $X_\alpha$ $(1- \varepsilon/2<\alpha\leq 1)$ and the radius of balls is $2\alpha$.
  \end{proof}

Then we will explain that for any unconditional basis of $L^p[0,1]$ and if it is monotone additionally, then the inequality (\ref{eq2}) is almost sharp. That is, the Haar basis is almost the best choice in the proof of Theorem \ref{LpBCP}.

\begin{defn}[\cite{AK2011,C1992,M1998}]
    Let $(e_n)_{n=1}^\infty$ be an unconditional basis of a Banach space $X$, then the suppression unconditional constant $K_{su}$ is the smallest real number such that for all  $J \subseteq \N$ the following inequality holds \[ \norm{\sum_{n \in J} e_n^*(x) e_n} \leq K_{su} \norm{x}.\]
\end{defn}

\begin{lem}[\cite{C1992}]\label{unc2}
 Let $1<p<\infty$ and $p^*:=\max(p,q)$, then for all $p^*>p_0$ where $p_0 \approx 2.5455458$ is the unique solution to
 \[p-2 = \left(\frac{(p-1)(p-2)}{-p^2+5p-5}\right)^{p-1},\]
 the suppression unconditional constant $K_{su}(p)$ is accurately
 \[K_{su}(p)=\frac{p^*}{2}+\frac{1}{2}\ln\left(\frac{1+e^{-2}}{2}\right)+\frac{\alpha_2}{p^*} +\cdots <\frac{p^*}{2},\]
 where $$\alpha_2=2^{-1}\ln \left(\frac{1+e^{-2}}{2}\right) + \left(2^{-1}\ln \left(\frac{1+e^{-2}}{2}\right)\right)^2 -2\left(\frac{e^{-2}}{1+e^{-2}}\right)^2.$$
 And if $p^* \leq p_0$, then $K_{su}(p) \leq p^*/2$.
\end{lem}

An unconditional basis is associated with its basis constant, unconditional basis constant and suppression unconditional constant, and it is worthwhile to know the relationship between these numbers. The following lemma give an order of the three important constants.

\begin{lem}[\cite{M1998}]
     Let $(e_n)_{n=1}^\infty$ be an unconditional basis of a Banach space $X$, $K_b$ be the basis constant, $K_u$ be the unconditional constant and $K_{su}$ be the unconditional suppression constant, then \[1\leq K_{b} \leq K_{su} \leq \frac{1+K_u}{2} \leq K_u \leq 2K_{su}.\]
\end{lem}

Then we can consider the lower bound of the unconditional constant of any monotone unconditional basis of $L^p[0,1]$ and by Lemma \ref{unc}, we know that the  Haar basis has almost the smallest unconditional constant.

\begin{cor}
    Let $(e_n)_{n=1}^\infty$ be a monotone unconditional basis of $L^p[0,1]$ for any $1<p <\infty$, then the unconditional constant $K_u(p)$ satisfies
    \[ K_u(p) \geq \max\left(p,\frac{p}{p-1}\right)-1+\varepsilon(p),\]
    where $$\varepsilon(p)=\chi_{(p_0,+\infty)}(p)\left(\ln\left(\frac{1+e^{-2}}{2}\right)+\frac{2\alpha_2}{\max\left(p,p/(p-1)\right)} +\cdots \right)\leq 0.$$
\end{cor}


\subsection*{Acknowledgment}
The authors are very grateful to Lixin Cheng for inspiring suggestions on ball-covering property and his invitation to visit Xiamen University. The authors would like to express their gratitude for visiting Institute for Advanced Study in Mathematics of HIT in the summer workshops of 2018, 2019, 2022 and 2023. The authors would like to thank  Minzeng Liu for helpful discussions. The authors would also like to express their appreciation to the referees for carefully reading the manuscript and providing many helpful comments and suggestions that helped improve the representation of this paper.


\begin{thebibliography}{40}
\bibitem{AK2011} Albiac, F., Kalton, N.J.: Topics in Banach Space Theory. Graduate Texts in Mathematics, vol. 233. Springer, Cham (2011)
    
\bibitem{AH2011} Argyros, S.A., Haydon, R.G.: A hereditarily indecomposable $\mathscr{L}_{\infty}$-space that solves the scalar-plus-compact problem. Acta Math. {\bf 206}, 1--54 (2011)

\bibitem{AG2023} Avil\'{e}s, A., Mart\'{\i}nez-Cervantes, G., Rueda Zocz, A.: A renorming characterisation of Banach spaces containing $\ell_{1}(k)$. Publ. Math. {\bf 67}, 601--609 (2023)

\bibitem{C2006} Cheng, L.: Ball-covering property of Banach spaces. Israel J. Math. {\bf 156}, 111--23 (2006)

\bibitem{CCL2008} Cheng, L., Cheng, Q., Liu, X.: Ball-covering property of Banach spaces that is not preserved under linear isomorphisms. Sci. China Ser. A (2008)

\bibitem{CKWZ2010} Cheng, L., Kadets, V., Wang, B., Zhang, W.: A note on ball-covering property of Banach spaces. J. Math. Anal. Appl. {\bf 371}(1), 249--253 (2010)

\bibitem{CKZ2020} Cheng, L., Kato, M., Zhang, W.: A survey of ball-covering property of Banach spaces. In: The mathematical legacy of Victor Lomonosov-operator theory, pp. 67--84. Adv. Anal. Geom. De Gruyter, Berlin (2020)

\bibitem{CLL2010} Cheng, L., Luo, Z., Liu, X., Zhang, W.: Several remarks on ball-coverings of normed spaces. Acta Math. Sin. {\bf 26}(9), 1667--1672 (2010)

\bibitem{CSZ2009} Cheng, L., Shi, H., Zhang, W.: Every Banach spaces with a $w^*$-separable dual has a $1+\varepsilon$-equivalent norm with the ball covering property. Sci. China Ser. A {\bf 52}, 1869--1874 (2009)

\bibitem{CWZ2011} Cheng, L., Wang, B., Zhang, W., Zhou, Y.: Some geometric and topological properties of Banach spaces via ball coverings. J. Math. Anal. Appl. {\bf 377}, 874--880 (2011)

\bibitem{C1992} Choi, K.P.: A sharp inequality for martingale transforms and the unconditional basis constant of a monotone basis. Trans. Amer. Math. Soc. (1992)

\bibitem{CLL2023} Ciaci, S., Langemets, J., Lissitsin, A.: A characterization of Banach spaces containing $\ell_{1}(k)$ via ball-covering properties. Israel J. Math. {\bf 253}, 359--379 (2023)

\bibitem{F2019} Farah, I.: Combinatorial Set Theory of C*-algebras. Cham: Springer (2019)

\bibitem{FZ2004} Fonf, V.P., Zanco, C.: Almost flat locally finite coverings of the
sphere. Positivity {\bf 8}(3), 269--281 (2004)

\bibitem{FZ12009} Fonf, V.P., Zanco, C.: Finitely locally finite coverings of Banach spaces. J. Math. Anal. Appl. {\bf 350}(2), 640--650 (2009)

\bibitem{FZ22009} Fonf, V.P., Zanco, C.: Coverings of Banach spaces: beyond the
Corson theorem. Forum Math. {\bf 21}(3), 533--546 (2009)

\bibitem{FZ32009} Fonf, V.P., Zanco, C.: Covering spheres of Banach spaces by balls.  Math. Ann. {\bf 344}, 939--945 (2009)

\bibitem{FR2016} Fonf, V.P., Rubin, M.: A reconstruction theorem for locally convex metrizable spaces, homeomorphism groups without small sets, semigroups of non-shrinking functions of a normed space. Topology Appl. {\bf 210}(1), 97--132 (2016)


\bibitem{HKL2023} Huang, J., Kudaybergenov, K., Liu, R.: The ball-covering property of noncommutative symmetric spaces. Studia Math. {\bf 277}(2), 169--190 (2024)

\bibitem{LT1977} Lindenstrauss, J., Tzafriri, L.: Classical Banach Spaces I: Sequence Spaces. Springer Science \& Business Media (1977)

\bibitem{LLLZ2022} Liu, M., Liu, R., Lu, J., Zheng, B.: Ball covering property from commutative function spaces to non-commutative spaces of operators. J. Funct. Anal.  {\bf 283}(1), 109502 (2022)

\bibitem{LLW2017} Luo, Z., Liu, J., Wang, B.: A remark on the ball-covering property of product spaces. Filomat {\bf 31}, 3905--3908 (2017)

\bibitem{LZ2020} Luo, Z., Zheng, B.: Stability of the ball-covering property. Studia Math. {\bf 250}, 19--34 (2020)

\bibitem{LZ2021} Luo, Z., Zheng, B.: The strong and uniform ball covering properties. J. Math. Anal. Appl. {\bf 499}, 125034 (2021)

\bibitem{M1998} Megginson, R.E.: An Introduction to Banach Space Theory. Graduate Texts in Mathematics, vol. 183. Springer, Cham (1998)

\bibitem{S2021} Shang, S.: The ball-covering property on dual spaces and Banach sequence spaces. Acta Math. Sci. Ser. B {\bf 41}(2), 461--474 (2021)

\bibitem{SC2013} Shang, S., Cui, Y.: Ball-covering property in uniformly non-$l_{3}^{(1)}$ Banach spaces and application. Abstr. Appl. Anal. {\bf 2013}(1), 1--7 (2013)

\bibitem{SC2015} Shang, S., Cui, Y.: Locally $2$-uniform convexity and ball-covering property in Banach space. Banach J. Math. Anal. {\bf 9}, 42--53 (2015)

\bibitem{SC2018} Shang, S., Cui, Y.: Dentable point and ball-covering property in Banach spaces. J. Convex Anal. {\bf 25}, 1045--1058 (2018)
    
\bibitem{T2012} Tarbard, M.: Hereditarily indecomposable, separable $\mathscr{L}_{\infty}$ Banach spaces with $\ell_1$ dual having few but not very few operators. J. London Math. Soc. {\bf 85}(3), 737--764 (2012)

\bibitem{Z2012} Zhang, W.: Characterizations of universal ﬁnite representability and B-convexity of Banach spaces via ball coverings. Acta Math. Sin. {\bf 28}, 1369--1374 (2012)
\end{thebibliography}
\end{document}